\documentclass[11pt]{amsart}

\usepackage{amsmath,amssymb,amsthm}
\usepackage[bookmarksopen=true,
            bookmarksnumbered=true,
            colorlinks=true,
            linkcolor=blue,
            pdftitle={Graded Betti numbers of powers},
            pdfsubject={Commutative Algebra},
            pdfkeywords={Asymptotic linearity, regularity, multigraded regularity, Betti number},
            pdfpagemode=UseOutlines,
            pdfproducer={Latex with hyperref},
            letterpaper=true,
            pdfcreator={latex->dvips->ps2pdf}]{hyperref}
\usepackage{graphicx,epsfig}
\usepackage{enumerate}
\usepackage{tikz}
\usepackage{multicol}
\usetikzlibrary{calc}
\usepackage{xy}
  \input{xy}
  \xyoption{all}
\usepackage{url}

\numberwithin{equation}{section}

\theoremstyle{plain}
\newtheorem{theorem}{Theorem}[section]

\newtheorem{corollary}[theorem]{Corollary}

\theoremstyle{definition}
\newtheorem{example}[theorem]{Example}

\newtheorem{definition}[theorem]{Definition}

\def\0{{\bf 0}}

\def\G{\mathbb{G}}

\def\NN{\mathbb{N}}

\def\R{\mathcal{R}}

\def\ZZ{\mathbb{Z}}
\def\RR{\mathbb{R}}

\def\QQ{\mathbb{Q}}

\def\m{\frak{m}}
\def\n{\frak{n}}

\DeclareMathOperator{\supp}{Supp}

\DeclareMathOperator{\tor}{Tor}

\DeclareMathOperator{\pos}{Pos}

\begin{document}

%\author{kamran lamei and Siamak Yassemi}

%\subjclass[2010]{Primary: 16G30, 16D90; Secondary: 16E10, 16S50}
%\keywords{Base change, Derived equivalence, Tilting class, Tilting module}
\author{Kamran Lamei \and Siamak Yassemi}
\address[Kamran Lamei]{School of Mathematics, Statistics and Computer Science \\ University of Tehran \\ Tehran \\ Iran}
%\thanks{}
\email{Kamran Lamei@ut.ac.ir}
\address[Siamak Yassemi]{School of Mathematics, Statistics and Computer Science \\ University of Tehran \\ Tehran \\ Iran}
%\thanks{}
\email[corresponding author]{yassemi@ut.ac.ir}

\title{graded Betti numbers of good filtrations}

%\today
\begin{abstract}
The asymptotic behavior of graded Betti numbers of powers of homogeneous ideals in a polynomial ring
over a field has recently been reviewed. We extend quasi polynomial behavior of graded Betti numbers of powers of homogenous ideals to $\ZZ$-graded algebra over Notherian local ring. Furthermore our main result treats the Betti table of  filtrations which is finite or integral over the Rees algebra.

\end{abstract}

\maketitle

%-----------------------------------------------------------------------------------------------------------------------%

\section{Introduction}
A significant result on the Castelnuovo-Mumford regularity of powers of homogeneous ideal $I$ in a polynomial ring $S$ shows that the maximal degree of the i-th syzygy of $I^t$ is a linear function of $t$ for $t$ large enough. This result have various generalizations including a case $I$ is a homogeneous ideal in  standard graded algebras over a Noetherian ring. Bagheri, Chardin and  H\`a promote the content in \cite{BCH} through investigating of the eventual behavior of all the minimal generators of the $i$-th syzygy module of $MI^t$ where $M$ is a finitely generated $Z$-graded S-module. Their approach formed from the fact that the module $\oplus_{t} \tor_{i}^{S} (MI^{t} , k )$ for a homogeneous ideal $I$ in graded ring $S$, has a structure of a finitely generated graded module over a non-standard graded polynomial ring over $k$, from which one can conclude the behavior of $\tor_{i}^{S} (I^{t} , k )$ when $t$ varies.\\

In the case that $I$ is generated by the forms of the same degree, an interesting result in \cite{BCH} shows that the Betti tables of modules $MI^t$ could be encoded by a variant of Hilbert-Serre polynomials. This is a refinement of asymptotic stability of total Betti numbers proved by Kodiyalam \cite{kodiyalam1993homological}. The question that track the asymptotic behavior of graded Betti numbers in the case that ideal $I$ is generated by forms of degrees (not necessarily equal) $d_1,\ldots ,d_r$  was the core of our previous work \cite{BK}. The fact that $\oplus_{t} \tor_{i}^{S} (I^{t} , k )$ is a finitely generated over a multigraded polynomial ring $B=K[T_1, \ldots,T_r]$ endowed $(d_i , 1)$ to each variables $T_i$ for $1\leq i \leq r$, guarantees a bigraded  minimal free $B$-resolution with free module $\oplus B(-a,-b)^{\beta_{(i,(a,b))}}$ at the homological degree $i$. As a consequence, we proved in \cite{BK} that $\ZZ^2$ can be splitted into a finite number of regions (see Figure \ref{regions}) such that each region  corresponds to quasi-polynomial behavior of the Betti numbers of homogeneous ideals in a polynomial ring over a field. Accordingly, the graded Betti tables of power $I^n$ could be encoded by the a set of polynomials for $n$ large enough. More generally on the stabilization of graded Betti numbers for a collection of graded ideals we use the fact that the module

$$
B_i:=\oplus_{t_1,\ldots ,t_s}\tor_i^R(MI_1^{t_1}\cdots I_{s}^{t_s},k)
$$

is a finitely generated $(\ZZ^p\times \ZZ^s)$-graded ring over $k[T_{i,j}]$ setting $\deg (T_{i,j})=(\deg (f_{i,j}),e_i)$
with $e_i$ the $i$-th canonical generator of $\ZZ^s$ and for fixed $i$ the $f_{i,j}$ form a set of minimal generators of  $I_i$. We proved in \cite{BK} that the chamber associated to  a positive $\ZZ^d$-grading of $R:=k[T_{i,j}]$ splitted $\ZZ^d$ into a finite maximal cells where the eventual behaviour of graded Betti numbers is encoded by the a set of polynomials.\\
\begin{figure}
\includegraphics[scale=0.20]{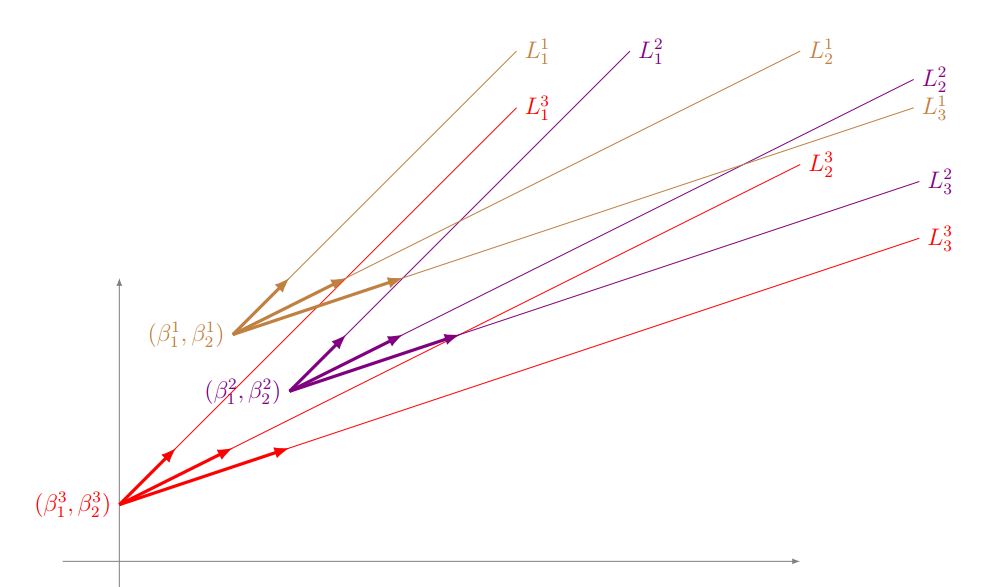}
\xymatrix{& \ar@{~>}[rrr]_{\mathrm{as\,\, degrees\,\,goes\,\, large\,\, enough}}^-{In\,\, the\,\, case \,of\, one\,\, graded\,\, ideal}&&&}
%\vspace{0.7cm}
\includegraphics[scale=0.20]{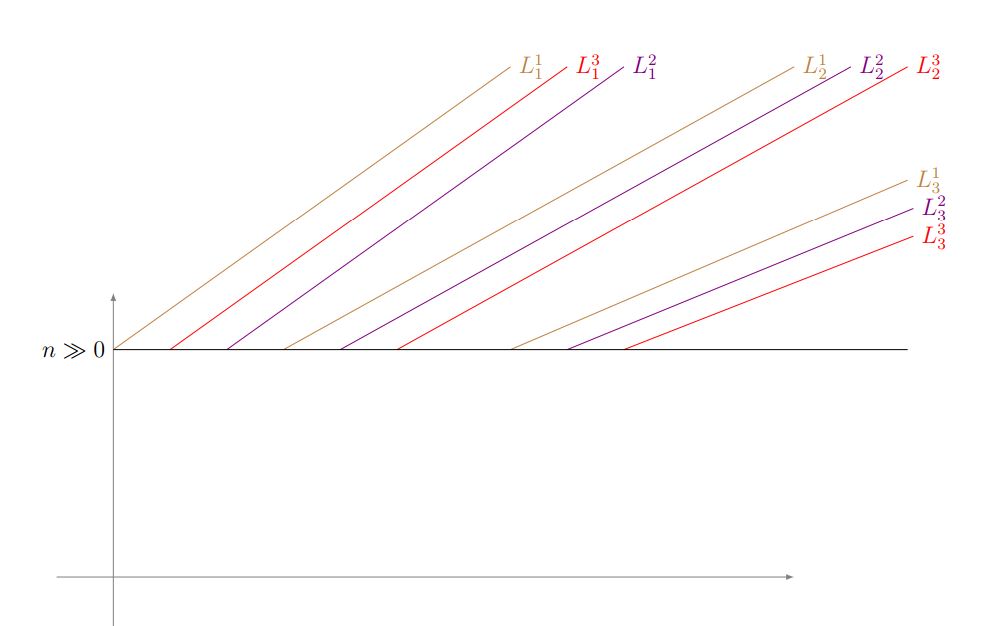}
\caption{}
\label{regions}
\end{figure}
Let $S=A[x_1,\ldots,x_n]$  be a graded algebra over a commutative noetherian local ring $A$ with residue field $k$ and $I \subseteq S$ be a homogeneous ideal.
% There is a ring homomorphism $A\rightarrow k$ bringing the following spectral sequence  $E^{2}_{p,q} = \tor_{p}(\tor_{q}^{S} (I^{t} , A)_{\nu},k) \Rightarrow \tor_{p+q}^{S} (I^{t} , k )_{\nu}$ in which one can write the Hilbert function of $\tor_{i}^{S} (I^{t} , k )_{\nu}$ in terms of Hilbert function of $\tor_{p}(\tor_{q}^{S} (I^{t} , A)_{\nu},k)$ for $i=p+q$. 
In Theorem \ref{garaded algebra} we generalize the quasi-polynomial behavior of the Betti numbers  to the case of homogeneous ideals in $G$-graded algebra over Notherian local ring. By using the concept of vector partition function we present an effective method to find such a set of polynomials.  roughly speaking, vector partition function is the number of ways where a vector decompose as a linear combination with nonnegative integral coefficients of a fixed set of vectors. On the other hand, let $e_i$ be the $i$-th standard basis of the space $\RR^r$ for $1\leqslant i \leqslant r$ and suppose that linear map $f : \RR^{n} \rightarrow \RR^d $ is defined by $f(e_i) = v_i$. The Rational convex polytope can be written as \\
$$
  P(b) := f^{-1} (b) \cap \RR_{\geqslant 0}^{r} = \{ x \in \RR^{r} | A x = a  ; x \geqslant 0 \}
$$
where $A$ is the matrix  of $f$. If $b$ is in the interior of  $\pos(A):=\{\sum_{i=1}^{n} \lambda_ia_i\in\RR^d | \lambda_i\geq 0, 1\leq i\leq n\}$, the polytope $P(b)$ has dimension $n-d$. One can see that evaluating the vector partition is equivalent to computing the number of integral points in a rational convex polytope. Accordingly in order to compute the Hilbert function $HF(S;b)$ of a polynomial rings $S$ over an infinite field $K$ of characteristic $0$ weighted by a set of vectors in $\ZZ^r$  we use an algorithm of enumeration the integral points in the polytope $P(b)$.

 A classical result of G.Pick [1899] for a two-dimensional polygone, states that if  $P \subset \RR^2$ is an integer polygon, then the number of integer points inside 
$P$ is
$
|P \cap \ZZ^2| = area (P) + \frac{| \partial P \cap \ZZ^2|}{2} + 1
$. One of important generalizations of the Pick's formula is the theorem of Ehrhart. In the general, for any rational polyhedron $P \subset \RR^n$ we consider following generationg function:
$$
f(P,\textbf{x}) = \sum_{m \in P \cap \ZZ^n} \textbf{x}^{m}
$$
where $m = (m_{1},\cdots ,m_{n})$ and $\textbf{x}^m = x_{1}^{m_1}\cdots x_{n}^{m_n}$. By Brion's theorem \cite{B}, the generating function of the polytope $P$ is equal to the sum of the generating functions of its vertex cones. More precisely, 
$$
f (P; \textbf{x}) = \sum_{m \in P \cap \ZZ^n} \textbf{x}^{m} = \sum_{v \in \Omega(P)} f (K; \textbf{v})
$$
where $\Omega (P)$ is the set of vertices of $P$. In order to find the generation function of arbitrary pointed cones, Stanley \cite{S} gives a triangulation of a rational cone into simplicial cones. Barvinok proved \cite{BP} that every rational polyhedral cone can be triangulated into unimodular cones.  The method of Barvinok enabling us to calculate the generating function of the above polytope depending on $b$, before that we should mention that the polytope $P(b)$ associated to the matrix $A$ is not full dimensional so to use the Barvinok metod we need to transform  $P(b)$ to polytope $Q$ which is full dimensional and the integer points of $Q$ are in one-to-one correspondence to the integer points of $P(b)$. In the example \ref{4-gen} we show the implementation of above mothod in order to find out an effective method for our objective.

Finally in order to be able to investigate the behavior of graded Betti numbers of integral closures and rattliff-Rush closure of powers of ideals, we use the notion of $I$-good filtrations. This is a series $\mathcal{J}=\lbrace \mathcal{J}_n \rbrace_{ n\geq 0}$ of ideals such that $\oplus_{n\geqslant 0} \mathcal{J}_n$ is a finite module over the Rees ring. Our result on graded Betti numbers of $I$-good filtrations takes the following form:

\textbf{Theorem.}
Let $S = A\left[x_{1}, \cdots , x_{n}\right]$ be a graded algebra over a Noetherian local ring $(A, m) \subset S_{0}$. Let $\mathcal{J}=\lbrace \mathcal{J}_n \rbrace_{ n\geq 0}$ be an $I$-good filtration of ideals $\mathcal{J}_n$ of $S$ and $\mathcal{J}_1 = (f_1,f_2,...,f_r)$    
with $\deg f_i = d_i$ be $\ZZ$-homogenous ideal in $S$, and let $R=S[T_1, \ldots, T_r]$ be a bigraded polynomial extension of $S$
with $\deg(T_i)=(d_i, 1)$ and $\deg (a) = (\deg (a)  , 0) \in \ZZ \times\{0\}$ for all $a \in S$. Then,\\
\begin{itemize}

\item For all $i$ \\

 $\tor_{i}^{R}(\R_{\mathcal{J}},k)$ is finitely generated $k[T_1, \ldots, T_r]$-module.\\

\item  There exist,
$t_0,m,D\in \ZZ$, linear functions $L_i(t)=a_i t+b_i$,
for $i=0,\ldots ,m$, with $a_i$ among the degrees
of the minimal generators of $I$ and $b_i\in \ZZ$, and polynomials $Q_{i,j}\in \QQ [x,y]$ for
$i=1,\ldots ,m$ and $j\in 1,\ldots ,D$, such that, for $t\geq t_0$,

(i) $L_i(t)<L_j(t)\ \Leftrightarrow\ i<j$,

(ii) If $\mu <L_0(t)$ or $\mu >L_m(t)$, then $\tor_i^S(\mathcal{J}_t, k)_{\mu}=0$.

(iii)  If $L_{i-1} (t)\leq \mu \leq L_{i}(t)$ and
$a_i t-\mu \equiv j\mod (D)$, then
$$
\dim_k\tor_i^S(\mathcal{J}_t, k)_{\mu}=Q_{i,j}(\mu ,t).\\\\
$$

\end{itemize}

It's worth mentioning that even in a simple situation, namely when $I$ is a complete intersection ideal quite many polynomials are involved to give the asymptotic behavior of the graded Brtti numbers of powers of $I$.

%------------------------------------------------------------------------------------------------------------------------%
\section{Preliminaries} \label{sec.prel}

In this section, we give a brief recall on necessary notations and terminology used in the article.
%For basic facts in commutative algebra, we refer the reader to \cite{bh}.

%\subsection{Hilbert series}

Let $S=k[x_1, \ldots, x_n]$ be a polynomial ring over a field $k$. Let $\G$ be an abelian group.  A $\G$-grading of $S$ is a morphism $\deg : \ZZ^n \longrightarrow \G $ and If  $G$ is torsion-free and $S_0=k$, the grading is positive. Criterions for positivity are  given in \cite[8.6]{ms}. When $\G=\ZZ^d$ and the grading is positive, (generalized) Laurent series are associated to finitely generated graded modules:

\begin{definition}
The Hilbert function of a finitely generated module $M$ over a positively graded polynomial ring is the map:
$$\begin{array}{llll}
HF(M; -):&\ZZ^d&\longrightarrow& \NN\\
&\mu&\longmapsto&\dim_k(M_\mu).
\end{array}$$
The Hilbert series of $M$ is  the Laurent series $$H(M;t)=\sum_{{\bf \mu}\in \ZZ^d}\dim_k(M_{\mu})t^{\mu}.$$
\end{definition}
%\begin{remark}

%By \cite[8.8]{ms},  if $S$ is positively graded by $\ZZ^d$,  then the semigoup $Q=\deg(\NN^n)$ can be embedded in $\NN^d$. Hence,  after such a change of embedding,
%the above Hilbert series are  Laurent series in the usual sense.
%\end{remark}

Let $M$ be a finitely generated $\ZZ^d$-graded $S$-module. It admits a
finite minimal graded free $S$-resolution
$$\mathbb{F}_\bullet: 0\rightarrow F_u \rightarrow \ldots \rightarrow F_1\rightarrow F_0\rightarrow M\rightarrow 0.$$
Writing $$F_i=\oplus_\mu S(-\mu)^{\beta_{i,\mu}(M)},$$ the minimality
shows that  $\beta_{i,\mu}(M)=\dim_k\left(\tor_i^S(M,k)\right)_\mu$,
as the maps of $\mathbb{F}_\bullet\otimes_S k$ are zero.
We also recall that the support of a $\ZZ^d$-graded module $N$ is
$$\supp_{\ZZ^d} (N):=\{\mu\in \ZZ^d | N_\mu\neq0\} .$$
%and use the abbreviated notations $\ZZ[t]:=\ZZ[t_1,\ldots ,t_d]$ for $t=(t_1,\ldots ,t_d)$ and $t^\mu :=t_1^{\mu_1}\cdots t_d^{\mu_d}$
%for $\mu =(\mu_1,\ldots ,\mu_d )\in \ZZ^d$.

%\begin{proposition}\label{Hilbert nonstandard}
%Let $S=k[x_1, \ldots, x_n]$ be a positively  $\ZZ^d$-graded polynomial ring over the field $k$.
%Then the followings hold.
%\begin{enumerate}
%\item
%The Hilbert series of a shifted copy $S(-\mu)$ of $S$  is the development in Laurent series of the rational function
  % $$H(S(-\mu ); t)= \frac{t^\mu}{\prod_{i=1}^n(1-t^{\mu_i})}.$$
%where $\mu_i = \deg (x_i)$.\\

%\item If $M$ is a finitely generated graded $S$-module, setting $\Sigma_M:=\cup_{\ell}\supp_{\ZZ^d}(\tor^R_\ell (B,k))$ and
%$$
%\kappa_M (t):=\sum_{a\in \Sigma_M}\left( \sum_\ell (-1)^\ell \dim_k (\tor^R_\ell (B,k))_a\right) t^a,
%$$
%one has $H(M;t) = \kappa_M (t)H(S; t)$.
%\end{enumerate}
%\end{proposition}
%----------------------------------------------------------------------------------%

\subsection{Partition function}
Let $e_i$ be the standard basis of the space $\RR^r$ for $1\leqslant i \leqslant r$. Let $f$ be a linear map 
$f : \RR^{r} \rightarrow \RR^d $ defined by $f(e_i) = v_i$ and denote by $V$ the linear span of $\left[ v_1, \cdots , v_r\right]$. For $a \in V$ consider the following convex polytope: \\
$$
  P(a) := f^{-1} (a) \cap \RR_{\geqslant 0}^{r} = \{ x= (x_1, \cdots , x_r) \in \RR^{r} |  \sum_{i=1}^{r} x_{i}v_{i} = a  ; x \geqslant 0 \}.
$$
%where $A$ is the matrix  of $f$ .  
\begin{definition}\label{partition function-def}
The function $\varphi : \NN^d \rightarrow \NN$ defined by $\varphi_{A}(a) = \sharp ( f^{-1} (a) \cap \ZZ_{\geqslant 0}^{r})$ is called vector partition function corresponding to the matrix $A = (v_1, \cdots , v_r)$.\\
\end{definition}

For more details about the vector partition functions in particular the definition of chambers, chamber complex and quasipolynomials we use the terminology of  \cite{BV, Strm}.

 Now, we recall the vector partition function theorem:
 %which relies on the chamber decomposition of $\pos(A)\subseteq \NN^d$.

 \begin{theorem}\label{vector partition}(See \cite[Theorem 1]{Strm})
 For each chamber $C$ of maximal dimension in the chamber complex of $A$,
 there exist a polynomial $P$ of degree
 $n-d$, a collection of polynomials  $Q_\sigma$
 and  functions $\Omega_\sigma: G_\sigma\setminus\{0\}\rightarrow \QQ$ indexed by non-trivial $\sigma\in \Delta(C)$
such that, if $u\in \NN A\cap \overline{C}$,
 $$\varphi_A(u)=P(u)+\sum\{\Omega_\sigma([u]_\sigma).Q_\sigma(u) : \sigma\in \Delta(C) , [u]_\sigma\neq 0\}$$
where $[u]_\sigma$ denotes the image of $u$ in $G_\sigma$. Furthermore, $\deg (Q_\sigma )=\#\sigma-d$.\\
\end{theorem}

\begin{corollary}\label{vector partition2}
 \cite{BK}For each chamber $C$ of maximal dimension in the chamber complex of $A$,
 there exists a collection of polynomials  $Q_\tau$ for $\tau\in \ZZ^d/\Lambda$ such that
 $$
 \varphi_A(u)=Q_\tau (u), \ \hbox{if}\ u\in \NN A\cap \overline{C}\ \hbox{and}\ u \in \tau+\Lambda_{C},
 $$
 where $\Lambda_{C} =  \cap_{\sigma \in \Delta(C)} \Lambda_{\sigma}$
\end{corollary}
%\begin{proof}
%The class $\tau$ of $u$ modulo $\Lambda$ determines  $[u]_\sigma$  in
%$G_\sigma=\ZZ^d/\Lambda_\sigma$. The term of the right-hand side of the equations in the above theorem is
%a polynomial determined by $[u]_\sigma$, hence by $\tau$.
%\end{proof}

Notice that setting $\Lambda$ for the intersection of the lattices
$\Lambda_\sigma$ with $\sigma$ maximal, the class of $u$ mod $\Lambda$  determines the class of $u$ mod $\Lambda_{C}$, hence the corollary holds with $\Lambda$ in place of $\Lambda_{C}$.\\

%-----------------------------------------------------------------------------------%
\section{Structure of  $\tor$ module of Rees algebra}

Let $S=A[x_1, \ldots, x_n]$ be a graded algebra over a commutative noetherian local 
ring $ S_0= (A,\textit{m})$ with residue field $k$ and set  $R=S[T_1, \ldots, T_r]$ and $B=k[T_1, \ldots, T_r]$ . We set $\deg(T_i)=(d_i, 1)$ and extended the grading from $S$ to $R$ by setting $\deg(x_i) = (\deg(x_i) , 0)$.
Let $M$ be a finitely generated graded $S$-module and $I$ be a graded $S$-ideal generated in degrees $d_1,\cdots d_r$. In this section we use the important  fact which it provides a $B$-structure on $\oplus_{t} \tor_{i}^{S} (MI^{t} , k )$ that were already at the center of the work  \cite{BCH} .

\begin{theorem}\label{garaded algebra}
Let  $S=A[x_1, \ldots, x_n]$ be a $\ZZ$-graded algebra over Noetherian local ring $(A,m,k)$. Let $I = (f_1,f_2,...,f_r)$ a homogeneous ideal in $S$ with $deg f_i = d_i$, and let $R=S[T_1, \ldots, T_n]$ be a bigraded polynomial extension of $S$ with $\deg(T_i)=(d_i, 1)$ and $\deg (a) = (\deg (a)  , 0) \in \G\times\{0\}$ for all $a \in S$. Then there exists,
$t_0,m,D\in \ZZ$, linear functions $L_i(t)=a_i t+b_i$,
for $i=0,\ldots ,m$, with $a_i$ among the degrees
of the minimal generators of $I$ and $b_i\in \ZZ$, and polynomials $Q_{i,j}\in \QQ [x,y]$ for
$i=1,\ldots ,m$ and $j\in 1,\ldots ,D$, such that, for $t\geq t_0$,

(i) $L_i(t)<L_j(t)\ \Leftrightarrow\ i<j$,

(ii) If $\mu <L_0(t)$ or $\mu >L_m(t)$, then $\tor_i^S(I^t, k)_{\mu}=0$.

(iii)  If $L_{i-1} (t)\leq \mu \leq L_{i}(t)$ and
$a_i t-\mu \equiv j\mod (D)$, then
$$
\dim_k\tor_i^S(I^t, k)_{\mu}=Q_{i,j}(\mu ,t).
$$
\end{theorem}

\begin{proof}
 The natural anto map $R \rightarrow \R_{I}:=\bigoplus_{t\geqslant 0} I^{t}$ sending $T_i$ to $f_i$ makes $\R_{I}=\bigoplus_{t\geqslant 0} I^{t}$ a finitely generated graded $R$-module. It was shown in \cite{BCH} that $\tor_{i}^{S}( I^{t} , A )$ is finitely generated $k[T_1, \ldots, T_r]$-module. Let $f: S \rightarrow A$ be the canonical map then there is a graded spectral sequence  with second term $E_{p,q}^{2} = \tor_{p}^{A}(\tor_{q}^{S} (I^{t} , A)_{\nu},k) \Rightarrow \tor_{p+q}^{S} (I^{t} , k )_{\nu}$ therefore $\tor_{i}^{S} (I^{t} , k )$ is finitely generated $k[T_1, \ldots, T_r]$-module. then the result follows from \cite[Proposition 4.5]{BK}.\\
\end{proof}

\begin{example}\label{4-gen}
Let $S=A[x_1,\ldots,x_n]$ be a graded algebra over a commutative noetherian local ring $S_0 =(A,m)$. Let $I\subseteq S$ be a complete intersection ideal of three forms $f_1,f_2,f_3,f_4$ of degrees $3,5,7,9$ respectively. Let $R=S[T_1,T_2,T_3,T_4]$ be a $\ZZ \times \ZZ$-graded polynomial extension of $S$ with $\deg x_i = (1,0)$ and $\deg (T_1)=(3,1)$, $\deg (T_2)=(5,1)$, $\deg (T_3)=(8,1)$ ,$\deg (T_1)= (9,1)$. By Theorem \ref{garaded algebra} $\tor_{i}^{R}(\R_{I},k )$ is finitely generated $B=k[T_1, \ldots, T_4]$-module with assigned  weight. Let $A= \begin{pmatrix}
3 & 5 & 8 & 9 \\
1 & 1 & 1 & 1
\end{pmatrix}$ be the matrix of degrees of $B$ then by  \cite[Proposition 3.1]{BK} the Hilbert function of $B$ at degree $(\nu,n)$ equals to enumeration  lattice points of the following  convex polytope
$$
  P(\nu,n) = \{ x \in \RR^{r} | A x = (\nu,n)  ; x \geqslant 0 \}
$$
 To take advantage of the Barvinok metod  of enumeration  lattice points we need a tranformation of polytope with  one-to-one correspondence to the lattice points by the following procedure explained in \cite{Lo}
 \begin{enumerate}
\item let $P = \{ x\in \RR^n | Ax = a, Bx \leqslant b\}$ be a polytope related to full row-rank $d\times n$ matrix $A$.
\item Find the generators $\{g_1,\cdots , g_{n-d}\}$ of the integer null-space of $A$.
\item Find integer solution $x_0$ to $Ax = a$.
\item Substituting the general integer solution $x = x_{0} + \sum_{i=1}^{n-d}\beta_{i}g_{i}$ into the inequalities $Bx \leqslant b$.
\item By Substitution of (4) we arrive at a new system $C\beta \leqslant c$ which  defines the new polytope $Q = \{\beta \in \RR^{n-d} | C\beta \leqslant c \}$.\\
\end{enumerate}
In order to find out the integer null-space of $A$ we first calculate the Hermite
Normal form (HNF) of $A$ which is
$$
H:=HFN(A)=
\begin{pmatrix}
1 & 0 & 0 & 0 \\
0 & 1 & 0 & 0
\end{pmatrix}
$$
and 
$$
U=
\begin{pmatrix}
1 & 4 & 3 & 2 \\
-2 & -4 & -5 & -3 \\
1 & 1 & 2 & 0 \\
0 & 0 & 0 & 1 
\end{pmatrix}
$$

Here $U$ is a unimodular matrix such that $H = AU$. Then the columns of $U_{1}=\begin{pmatrix}
1 & 4 \\
-2 & -4 \\
1 & 1 \\
0 & 0
\end{pmatrix}$
gives the generators of the integer null-space of $A$. Hence by the above procedure the polytope $Q(\nu,n)$ defined by the solutions of following  system of linear inequalities
$$\left\{\begin{array}{r}
-4\lambda_{2}  -4n-\nu \leq \lambda_{1},\\
\lambda_{1} + 2\lambda_{2} \leq -\nu-2n,\\
\lambda_{1}+\lambda_{2} \geq -\nu-n ,\\
\end{array}\right.$$

\end{example}

\section{structure of  $\tor$ module of Hilbert filtrations}
To study blowup algebras, Northcott and Rees defined the notion of reduction of  an ideal $I$ in a commutative ring $R$. An ideal $J \subseteq I$ is a reduction of $I$ if there exists $r$ such that  $JI^{r} = I^{r+1}$ (equivalently this hold for $r \gg 0$) . An important fact about reduction of ideals is that this property is equivalent to the fact that 
$$
\R_{J} = \oplus_{n} J^{n} \rightarrow \R_{I} = \oplus_{n} I^{n}
$$
is a finite morphism. Okon and Ratliff in \cite{OR} extended the above notion of reduction to the case of filtrations by setting the following definition :\\

\begin{definition}
Let $R$ be a ring, $I$ an $R$-ideal and $\mathcal{J}=\lbrace \mathcal{J}_n \rbrace_{ n\geq 0}$ and $\mathcal{I}=\lbrace \mathcal{I}_n \rbrace_{ n\geq 0}$ two filtration on $R$ :

\item  (1) $\mathcal{J} \leq \mathcal{I}$ if $\mathcal{J}_n \subseteq \mathcal{I}_n$ for all $n\geq 0$.

\item  (2) $\mathcal{J}$ is a reduction of $\mathcal{I}$ if $\mathcal{J} \leq \mathcal{I}$ and there exists a positive integer $d$ such that $\mathcal{I}_n = \sum_{i=0}^{d} \mathcal{J}_{n-i}\mathcal{I}_i$ for all $n\geq 1$.

\item  (3  $\mathcal{J}$ is a $I$-good filtration if $I\mathcal{J}_i \subseteq \mathcal{J}_{i+1}$ for all $i\geq 0$
and $\mathcal{J}_{n+1} = I\mathcal{J}_n$ for all $n\gg 0$.

%\item  (5) Let $\gamma$ be $I$-good filtration , then a $J$-good filtration $\varphi$ is called good reduction of $\gamma$  if it is a reduction in the sense of (3).\\
\end{definition}

Opposite to the ideal case, minimal reductions of a filtration does not exist in general. But 
Hoa  and Zarzuela showed in \cite{HZ} the existence of a minimal reduction for  $I$-good filtrations.% as follows :

If $\mathcal{J}=\lbrace \mathcal{J}_n \rbrace_{ n\geq 0}$  is an $I$-good filtration on $R$, then $\R_\mathcal{J} := \oplus_{n\geqslant 0} \mathcal{J}_n$  is a finite $\R_{I}$-module \cite[Theorem III.3.1.1]{}. Thit is why we are interested about $I$-good filtration to generalize the previous results.
The following theorem explain the structure of  $\tor$ module of $I$-good filtrations :\\

\begin{theorem}\label{Tor-Hilbert filtrations}
Let $S = A\left[x_{1}, \cdots , x_{n}\right]$ be a graded algebra over a Noetherian local ring $(A, m,k) \subset S_{0}$ . Let $\mathcal{J}=\lbrace \mathcal{J}_n \rbrace_{ n\geq 0}$ be an $I$-good filtration of $\ZZ$-homogeneous  ideals in $S$, and $\mathcal{J}_1 = (f_1,f_2,...,f_r)$    
with $deg f_i = d_i$ . Let $R=S[T_1, \ldots, T_n]$ be a bigraded polynomial extension of $S$
with $\deg(T_i)=(d_i, 1)$ and $\deg (a) = (\deg (a)  , 0) \in \ZZ \times\{0\}$ for all $a \in S$.\\

\item(1)Then for all $i$ : \\

 $\tor_{i}^{R} (\R_{\mathcal{J}},k) $ is a finitely generated $k[T_1, \ldots, T_r]$-module .\\

\item (2) There exist,
$t_0,m,D\in \ZZ$, linear functions $L_i(t)=a_i t+b_i$,
for $i=0,\ldots ,m$, with $a_i$ among the degrees
of the minimal generators of $I$ and $b_i\in \ZZ$, and polynomials $Q_{i,j}\in \QQ [x,y]$ for
$i=1,\ldots ,m$ and $j\in 1,\ldots ,D$, such that, for $t\geq t_0$,

(i) $L_i(t)<L_j(t)\ \Leftrightarrow\ i<j$,

(ii) If $\mu <L_0(t)$ or $\mu >L_m(t)$, then $\tor_i^S(\varphi(t), k)_{\mu}=0$.

(iii)  If $L_{i-1} (t)\leq \mu \leq L_{i}(t)$ and
$a_i t-\mu \equiv j\mod (D)$, then
$$
\dim_k\tor_i^S(\mathcal{J}_t, k)_{\mu}=Q_{i,j}(\mu ,t).\\\\
$$

\end{theorem}

\begin{proof}

Recall that $\R_\mathcal{J}$ is a finite $\R_I$-module, hence a finitely generated $\ZZ^2$-graded $R$-module. Let $F_{\bullet}$ be a  $\ZZ \times \ZZ$-graded minimal free resolution of $\R_{\mathcal{J}}$ over $R$.  Each $F_i = \oplus_{\mu ,t} R (-\mu ,- n)^{\beta_{\mu ,n}^{i}}$ is of  of finite rank due to the Noetherianity of $A$. The graded stand $F_\bullet^t := (F_\bullet)_{\ast,t}$ is a $\ZZ$-graded free resolution of $\mathcal{J}_t$ over $S = R_{(*,0)}$. Meanwhile there is a $S$-graded isomorphisme
$$
 R (-\mu ,- n)_{(*,t)} \simeq  R (-\mu)_{(*,t-n)}
$$

 Thus  $\tor^S_i(\mathcal{J}_t,A) = H_i(F_\bullet^t \otimes_S A)$ and by the above isomorphisme $\tor^S_i(\mathcal{J}_t,A)$ is subquotient of $ A(-\mu)^{\beta_{\mu ,n}^{i}} \otimes_A \left( A[T_1, \ldots, T_r]\right) _{t-n} $ Since the $A$ is Noetherian it it follows that $\tor^S_i(\mathcal{J}_t,A)$ is a finitely generated $A[T_1, \ldots, T_r]$-module. With a similar Approach one has $\tor^S_i(\mathcal{J}_t,k) = H_i(F_\bullet^t \otimes_S k).$
Moreover, taking homology respects the graded structure and therefore,
$$H_i(F_\bullet^t \otimes_S k) = H_i(F_\bullet \otimes_R R/\m +\n R)_{(*,t)},$$
where $\n = (x_1, \dots, x_n)$ is the homogeneous irrelevant ideal of $S$. It follows that  $\tor_{j}^{R}(\R_{\mathcal{J}},k) $ is a finitely generated  graded $k[T_1, \ldots, T_r]$-module. 
To proof the second fact let $E:=\{ d_1,\ldots ,d_r\}$ with $d_1<\cdots <d_r$ be a set of positive integers. For $\ell$ from
$1$ up to $r-1$, let $$\Omega_\ell :=\{ a{{d_\ell}\choose{1}}+b{{d_{\ell +1}}\choose{1}},\ (a,b)\in \RR_{\geq 0}^2\} $$ be the closed cone spanned by ${{d_\ell}\choose{1}}$ and ${{d_{\ell +1}}\choose{1}}$.  For integers $i\neq j$, let $\Lambda_{i,j}$ be the lattice spanned by ${{d_i}\choose{1}}$ and ${{d_{j}}\choose{1}}$ and $
\Lambda_\ell :=\bigcap_{i\leq \ell < j}\Lambda_{i,j}.
$
Also we set $\Lambda :=\bigcap_{i < j}\Lambda_{i,j}$ with $\Delta = \det(\Lambda)$. It follows from Theorem \ref{vector partition} that $\dim_k B_{\mu ,t}=0$ if $(\mu ,t)\not\in \Omega:= \bigcup_\ell \Omega_\ell$. Part (ii) and (iii) then follow from \cite[Lemma 4.4 and Proposition 4.5]{BK}.\\\
\end{proof}

This in particular applies to the following situations: 

\begin{itemize}
\item  $I$ is a graded ideal of $S$ and $S$ is an analytically unramified ring without
nilpotent elements. Then the integral closure filtration $\mathcal{J}=\lbrace \overline{I^n}\rbrace$ is $I$-good filtration\cite{rees}.\\

\item  $I$ is a graded ideal of $S$, then the rattliff-Rush closure filtration $\mathcal{J}=\lbrace \widetilde{I^{n}}\rbrace$ is an $I$-good filtration\cite{rr}.

\end{itemize}

\section{Acknowledgement}
We would like to express our sincere gratitude to Marc Chardin for his support and help. The first author also acknowledge the support of Iran's National Elites Foundation (INEF). We thank Shahab Rajabi for comments and suggestions.
%----------------------------------------------------------------------------------------------------------------%

\bibliographystyle{amsplain}
\bibliography{thebib}{}

\end{document}